\documentclass[12pt]{amsart}

\usepackage{amsmath, amsthm, amssymb, enumerate, amscd}
\usepackage[mathscr]{eucal}
\usepackage{pgf,tikz}
\usepackage{tikz, tikz-cd}
\usetikzlibrary{matrix,decorations.pathmorphing,arrows}

\theoremstyle{plain}
\newtheorem{theorem}{Theorem}[section]
\newtheorem{proposition}[theorem]{Proposition}
\newtheorem{corollary}[theorem]{Corollary}

\theoremstyle{definition}
\newtheorem{definition}[theorem]{Definition}

\newcommand{\N}{{\mathbb{N}}}
\newcommand{\Z}{{\mathbb{Z}}}
\newcommand{\Q}{{\mathbb{Q}}}
\newcommand{\R}{{\mathbb{R}}}

\newcommand{\m}{{\mathfrak{m}}}
\newcommand{\n}{{\mathfrak{n}}}

\begin{document}

\title[A characterization of PIDs] {A new characterization of principal ideal domains}

\author{Katie Christensen}
\address{Department of Mathematics\\ 
University of Louisville\\
Louisville, KY 40292, USA}
\email{katie.christensen@louisville.edu}
\author{Ryan Gipson}
\address{Department of Mathematics\\ 
University of Louisville\\
Louisville, KY 40292, USA}
\email{ryan.gipson@louisville.edu}
\thanks{$\ast$ the corresponding author}
\author{Hamid Kulosman$^\ast$}
\address{Department of Mathematics\\ 
University of Louisville\\
Louisville, KY 40292, USA}
\email{hamid.kulosman@louisville.edu}



\date{}

\begin{abstract} 
In 2008 N.~Q.~Chinh and P.~H.~Nam characterized principal ideal domains as integral domains that satisfy the follo\-wing two conditions: (i) they are unique factorization domains, and (ii) all maximal ideals in them are principal. We improve their result by giving a characterization in which each of these two conditions is weakened. At the same time we improve a theorem by P.~M.~Cohn which characterizes principal ideal domains as atomic B\'ezout domains. We will also show that every PC domain is AP and that the notion of PC domains is incomparable with the notion of pre-Schreier domains (hence with the notions of Schreier and GCD domains as well). 
\end{abstract}

\maketitle

\section{Introduction and preliminaries}\label{intro}
The goal of this paper is to improve the 2008 result of N.~Q.~Chinh and P.~H.~Nam \cite[Corollary 1.2.]{cn} in which they gave a characterization of principal ideal domains as integral domains that satisfy the follo\-wing two conditions: (i) they are unique factorization domains, and (ii) all maximal ideals in them are principal. Our main result is a new charac\-terization of principal ideal domains obtained by weakenning each of the conditions in the Chinh and Nam's result. At the same time we improve the so called Cohn's theorem which characterizes principal ideal domains as atomic B\'ezout domains. In order to state our improvement, we will introduce a new condition for integral domains and prove that that new condition is indeed weaker than the corresponding conditions in the two mentioned theorems.

\medskip
We begin by recalling some definitions and statements. All the notions that we use  but not define in this paper can be found in the classical reference books \cite{c_book} by P.~M.~Cohn, \cite{g} by R.~Gilmer, \cite{k} by I.~Kaplansky, and \cite{n} by D.~G.~Northcott. 

\medskip
 In this paper all rings are {\it integral domains}, i.e., commutative rings with identity in which $xy=0$ implies $x=0$ or $y=0$.
A non-zero non-unit element $x$ of an integral domain $R$ is said to be {\it irreducible} (and called an {\it atom}) if $x=yz$ with $y,z\in R$ implies that $y$ or $z$ is a unit. A non-zero non-unit element $x$ of an integral doman $R$ is said to be {\it prime} if $x\mid yz$ with $y,z\in R$ implies $x\mid y$ or $x\mid z$. Every prime element is an atom, but not necessarily vice-versa. Two elements $x,y\in R$ are said to be {\it associates} if $x=uy$, where $u$ is a unit. We then write $x\sim y$. 

\medskip
An integral domain  $R$ is said to be {\it atomic} if every non-zero non-unit element of $R$ can be written as a (finite) product of atoms.  An integral domain $R$ is called a {\it principal ideal domain} (PID) if every ideal of $R$ is principal. The condition for integral domains that every ideal is principal is called the {\it PID condition}. An integral domain $R$ is called a {\it unique factorization domain} (UFD) if it is atomic and for every non-zero, non-unit $x\in R$, every two factorizations of $x$ into atoms are equal up to order and associates. An integral domain $R$ is called an {\it ACCP domain} if every increasing sequence of principal ideals of $R$ stabilizes. It is well-known that every PID is a UFD, every UFD is an ACCP domain, and every ACCP domain is atomic. 

\medskip
An integral domain $R$ is called a {\it B\'ezout domain} if every two-generated ideal of $R$ is principal. (An ideal $I$ of $R$ is said to be {\it two-generated} if $I=(a,b)$ for some $a,b\in R$.) The condition for integral domains that every two-generated ideal is principal is called the {\it B\'ezout condition}. Obviously, every PID is a B\'ezout domain. The converse is not true.

\begin{proposition}[{\cite[9.4, Exercise 5, pages 306-307]{df}}]\label{Bezou_example}
B\'ezout condition for integral domains is strictly weaker than the PID condition. More concretely, $R=\Z+X\Q[X]$ is a B\'ezout domain which is not a PID. 
\end{proposition}

Note that the notation  $R=\Z+X\Q[X]$ means that $R$ consists of all the polynomials from $\Q[X]$ whose constant term is from $\Z$. 

\medskip
We call the {\it PIP condition} the condition for integral domains that every prime ideal is principal. We call the {\it MIP condition} the condition for integral domains that every maximal ideal is principal. The {\it MIP domains} are the domains which satisfy the MIP condition. Clearly, the PID condition implies the PIP condition and the PIP condition implies the MIP condition. More precise relations between these conditions are given in the next proposition and Corollary \ref{PIP_MIP}.

\begin{proposition}[{\cite[8.2, Exercise 6, page 283]{df}}]\label{pid_pip}
The PID condition for integral domains is equivalent to the PIP condition. In other words, if every prime ideal of an integral domain $R$ is principal, then $R$ is a PID. 
\end{proposition}

\medskip
The final item that we cover in this introduction is the notion of a {\it monoid ring} for a commutative monoid $M$, written additively. The elements of the monoid ring $F[X;M]$, where $F$ is a field and $X$ is a variable, are the {\it polynomial expressions}, also called {\it polynomials}, 
\begin{equation}\label{polynomials}
f(X)=a_1X^{\alpha_1}+\dots+a_nX^{\alpha_n},
\end{equation}
where $n\ge 0$, $a_1,\dots, a_n\in F$,\, $\alpha_1, \dots, \alpha_n\in M$.  The polynomials $f(X)=a$, $a\in F$, are called the {\it constant polynomials}. The addition and the multiplication of the polynomials are naturally defined. We say that $M$ is {\it cancellative} if for any elements $a,b,c\in M$,  $a+b=a+c$ implies $b=c$. The monoid $M$ is {\it torsion-free} if for any $n\in\N$ and $a,b\in M$, $na=nb$ implies $a=b$. All the monoids that we use in this paper are cancellative and torsion-free, hence the monoid rings $F[X;M]$ are integral domains. 

\section{A new condition for integral domains}
We introduce a new condition for integral domains, that we haven't met in the literature.

\begin{definition}
We call the {\it principal containment condition} (PC) the condition for integral domains that every proper two-generated ideal is contained in a proper principal ideal. We say that an integral domain is a {\it PC domain} if it satisfies the PC condition.
\end{definition}

Clearly, B\'ezout condition implies the PC condition, and the MIP condition implies the PC condition. 

\begin{proposition}\label{mip_pc}
There exists a B\'ezout domain which is not a MIP domain.
\end{proposition}
\begin{proof}
Consider the monoid ring $R=F[X;\Q_+]$  ($F$ a field),  consisting of all the polynomials of the form
\begin{equation*}
f(X)=a_0+a_1X^{\alpha_1}+\dots+a_nX^{\alpha_n}
\end{equation*} 
with $a_0, a_1,\dots,a_n\in F$ and $0<\alpha_1<\dots<\alpha_n$ from $\Q_+$. 
 Let $\frak{m}$ be the maximal ideal of $R$ consisting of all the polynomials in $R$ whose constant term is $0$. Consider the localization $D=R_\frak{m}$. The units of $D$ have the form
\begin{equation*}
\frac{a_0+a_1X^{\alpha_1}+\dots +a_mX^{\alpha_m}}{b_0+b_1X^{\beta_1}+\dots +b_nX^{\beta_n}},
\end{equation*}
where the $a_i$ and $b_j$ are from $F$ with $a_0$, $b_0$ non-zero. Hence every non-zero element of $D$ has the form $uX^\alpha$, where $u$ is a unit in $D$ and $\alpha\in\Q_+$. The maximal ideal $\frak{m}R_\frak{m}$ of $D$ consists of all $uX^\alpha$ with $\alpha>0$ and is not finitely generated. So $D$ does not satisfy the MIP condition. However, for any two elements $uX^\alpha$, $vX^\beta$ of $D$ with $\alpha\le \beta$ we have $uX^\alpha\mid vX^\beta$ and so $D$ is B\'ezout.
\end{proof}

We will show later (see Proposition \ref{MIP_not_PS}) that
there also exists a MIP domain which is not a B\'ezout domain. Thus the notion of a PC domain is strictly weaker than each of the notions MIP and B\'ezout. Finally in Proposition \ref{PC_not_PS} we show that the notion of a PC domain is not ``just a union'' of the notions of B\'ezout and MIP domains, i.e., that there is a PC domain which is neither B\'ezout, nor MIP. 

\bigskip
Consider now the following diagram.

\bigskip\bigskip
\hspace{-2cm}
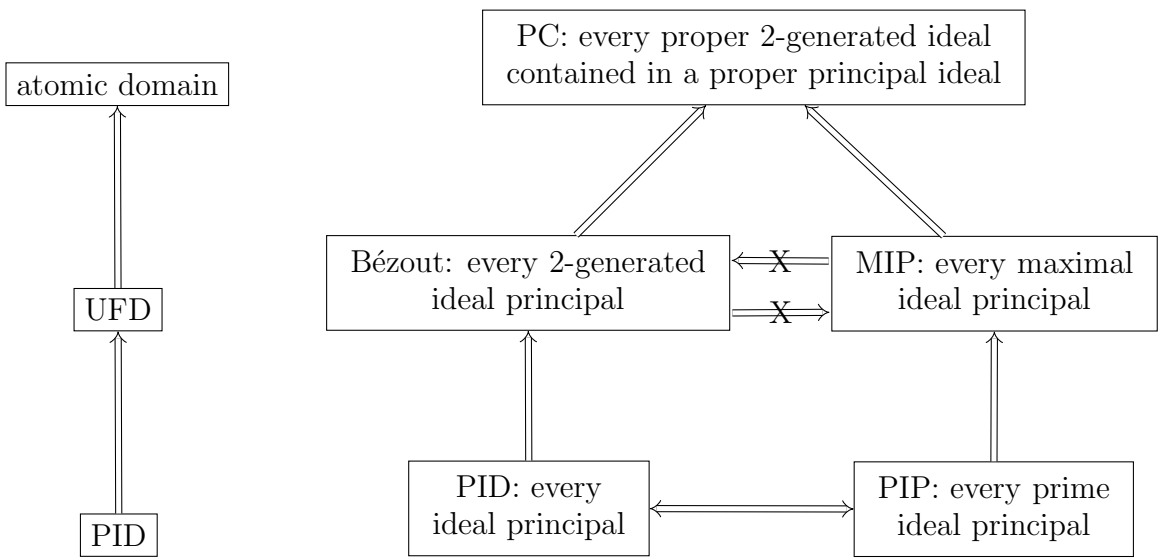
\begin{tabular}{cc}
\begin{tikzpicture}
\node[draw] (A) at (2.98,0) {atomic domain};
\node[draw] (UFD) at (2.99,-3) {UFD};
\node[draw] (PID) at (3,-6) {PID};
\draw[-implies, double equal sign distance] (PID) -- (UFD);
\draw[-implies, double equal sign distance] (UFD) -- (A);
\end{tikzpicture}
& \hspace{.8cm}
\begin{tikzpicture}
\node[draw] (PC) at (13,0) {\begin{tabular}{c} PC: every proper 2-generated ideal\\ contained in a proper principal ideal\end{tabular}};
\node[draw] (B) at (10,-3) {\begin{tabular}{c}B\'ezout: every 2-generated\\ ideal principal\end{tabular}};
\node (B_phantom1) at (12.56,-2.7) {};
\node (B_phantom2) at (12.57,-3.4) {};
\node[draw] (MIP) at (16.2,-3.01) {\begin{tabular}{c} MIP: every maximal\\ ideal principal\end{tabular}};
\node (MIP_phantom1) at (14.14,-2.71) {};
\node (MIP_phantom2) at (14.15,-3.39) {};
\node[draw] (PIDc) at (10.01,-6) {\begin{tabular}{c} PID: every\\ ideal  principal\end{tabular}};
\node[draw] (PIP) at (16.19,-6.01) {\begin{tabular}{c} PIP: every prime\\ ideal principal\end{tabular}};
\draw[-implies, double equal sign distance] (PIDc) -- (B);
\draw[implies-implies, double equal sign distance] (PIDc) -- (PIP);
\draw[-implies, double equal sign distance] (PIP) -- (MIP);
\draw[-implies, double equal sign distance] (MIP_phantom1) -- node{X} (B_phantom1);
\draw[-implies, double equal sign distance] (B_phantom2) -- node{X} (MIP_phantom2);
\draw[-implies, double equal sign distance] (B) -- (PC);
\draw[-implies, double equal sign distance] (MIP) -- (PC);
\end{tikzpicture}
\end{tabular}

\bigskip\bigskip\bigskip
There is one equivalence in the diagram, the rest are implications (and all of them are strict) and non-implications. The higher the condition is (in each of the two parts of the diagram), the weaker it is. One can try to characterize PIDs by combining one condition from the left part of the diagram with one condition from the right part of the diagram.

\medskip
The next two theorems are characterizations of PIDs of that type. The first one ({\it Cohn's Theorem}) is Theorem \ref{cohn_thm} that was first stated in \cite[Proposition 1.2]{c}. (Cohn remarks in \cite{c} that it is easy to prove that Bezout's domains which satisfy ACCP are PIDs, however, ACCP is not equivalent to atomicity, as it was later shown.) The proof can be seen in Cohn's book \cite[10.5 Theorem 3]{c_book}.  The second one is Theorem \ref{cn_thm}, proved in 2008 by Chinh and Nam in \cite{cn}. 

\begin{theorem}[Cohn's Theorem]\label{cohn_thm}
If $R$ is an atomic B\'ezout domain, then $R$ is a PID.
\end{theorem}

\begin{theorem}[{\cite[Corollary 1.2.]{cn}}]\label{cn_thm}
If $R$ is a UFD in which every maximal ideal is principal, then $R$ is a PID.
\end{theorem}

Our next theorem improves both of the above theorems. It weakens one of the conditions in Cohn's theorem and both conditions in the Chinh and Nam's theorem.

\begin{theorem}\label{main}
Let $R$ be an atomic domain which satisfies the PC condition. Then $R$ is a PID.
\end{theorem}
\begin{proof}
By Proposition \ref{pid_pip}, it is enough to show that every prime ideal is principal. Let $P$ be a nonzero prime ideal of $R$. Let $x\ne 0$ be an element of $P$. Since $R$ is atomic, we can write $x=p_1p_2\cdots p_n$, where the $p_i$'s are atoms.  As $p_1p_2\cdots p_n\in P$ and $P$ is prime, at least one of the $p_i$'s, say $p_1$, is in $P$. We claim that $P=(p_1)$. Let $y$ be an element of $P$. Since $R$ satisfies PC, $(p_1, y)\subset (c)$ for some proper principal ideal $(c)$. From $p_1=ct$ for some $t\in R$, we have $t\sim 1$ (as $p_1$ is an atom and $c\nsim 1$). Now from $y=cr$ for some $r\in R$, we get $y=p_1t^{-1}r$, hence $y\in (p_1)$. Thus $P=(p_1)$.
\end{proof}

\section{Merging the diagrams and making them more detailed}

In this section we will merge the diagrams from the previous section and make them more detailed. That will illustrate the importance of the notion of a PC domain that we introduced in the previous section. We first need to give some definitions. 

\medskip
An integral domain is called a {\it GCD domain} if every two elements of it have a greatest common divisor (see \cite[page p.4]{a}). An element $c$ of an integral domain $D$ is called {\it primal} if for any $a,b\in D$ we have: $c\mid ab \Rightarrow c=c_1c_2$ where $c_1\mid a$ and $c_2\mid b$. This notion was introduced in \cite{c}, where a new version of the definition of Schreier domains is also given: an integral domain $D$ is {\it Schreier} if it is integrally closed and each of its elements is primal. The notion of pre-Schreier domains is introduced in \cite{z}: an integral domain is {\it pre-Schreier} if each of its elements is primal. Clearly every Schreier domains is pre-Schreier, but not conversely. A new proof of the well-known result that every GCD domain is Schreier was given in \cite{c}. The converse is not true. Also, every B\'ezout domain is GCD, but not conversely (see \cite{c}). An integral domain is called an {\it AP domain} if each of its atoms is prime, i.e., if the notions of an atom and of a prime element in it coincide. Every pre-Schreier domain is an AP domain, but not vice-versa (see \cite{z}). It is well-known that an integral domain is a UFD if and only if it is atomic and AP.

\bigskip\bigskip
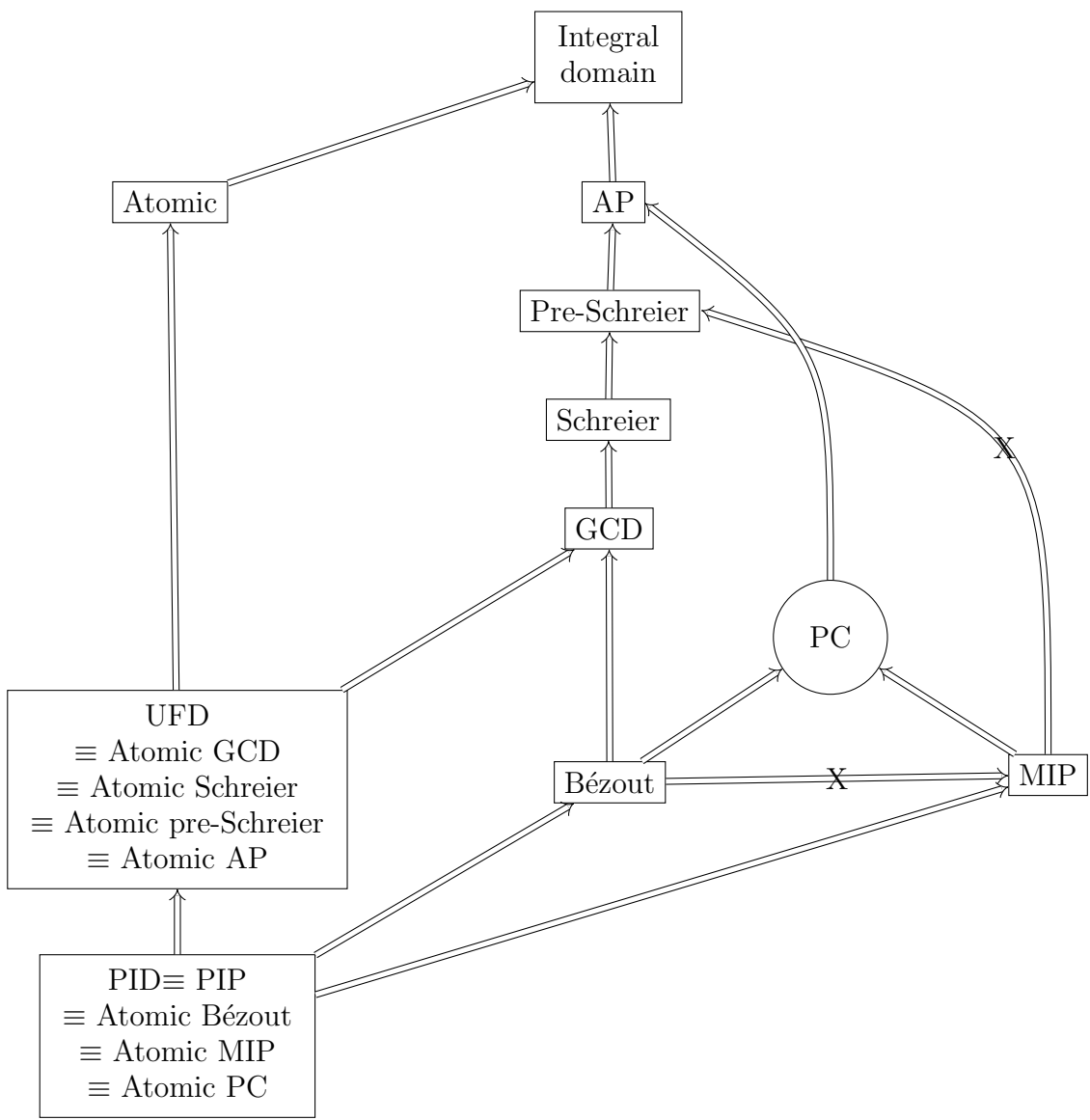
\hspace{-1.8cm}
\begin{tikzpicture}
\node[draw] (ID) at (10.94,0) {\begin{tabular}{c} Integral\\ domain\end{tabular}};
\node[draw] (AP) at (11.01, -2) {AP};
\node (phantom_AP) at (11.3, -1.9) {};
\node (phantom_PS) at (11.3, -3.2) {\phantom{Pre-Schreier}};
\node[draw] (PS) at (10.96,-3.5) {Pre-Schreier};
\node[circle, draw] (PC) at (14,-8) {\phantom{X}PC\phantom{X}};
\node[draw] (A) at (4.9,-2) {Atomic};
\node[draw] (S) at (10.94,-5) {Schreier};
\node[draw] (GCD) at (10.95, -6.5) {GCD};
\node[draw] (UFD) at (5, -10.1) {\begin{tabular}{c} UFD\\ $\equiv$ Atomic GCD\\  $\equiv$ Atomic Schreier\\ $\equiv$ Atomic pre-Schreier\\ $\equiv$ Atomic AP \end{tabular}};
\node[draw] (B) at (10.96,-10) {B\'ezout};
\node[draw] (MIP) at (17,-9.9) {MIP};
\node[draw] (PID) at (5,-13.5) {\begin{tabular}{c}PID$\equiv$ PIP\\ $\equiv$ Atomic B\'ezout\\ $\equiv$ Atomic MIP \\ $\equiv$ Atomic PC\end{tabular}};                                                                                                                                                                                                
\draw[-implies, double equal sign distance] (PID) -- (UFD);
\draw[-implies, double equal sign distance] (PID) -- (B);
\draw[-implies, double equal sign distance] (PID) -- (MIP);
\draw[-implies, double equal sign distance] (UFD) -- (GCD);
\draw[-implies, double equal sign distance] (B) -- (GCD);
\draw[-implies, double equal sign distance] (B) -- node{X} (MIP);
\draw[-implies, double equal sign distance] (B) -- (PC);
\draw[-implies, double equal sign distance] (MIP) -- (PC);
\draw[-implies, double equal sign distance] (MIP) .. controls (17,-5) ..  node{X} (phantom_PS);
\draw[-implies, double equal sign distance] (UFD) -- (A);
\draw[-implies, double equal sign distance] (GCD) -- (S);
\draw[-implies, double equal sign distance] (S) -- (PS);
\draw[-implies, double equal sign distance] (PC) ..controls (14,-4).. (phantom_AP);
\draw[-implies, double equal sign distance] (PS) -- (AP);
\draw[-implies, double equal sign distance] (AP) -- (ID);
\draw[-implies, double equal sign distance] (A) -- (ID);
\end{tikzpicture}

\bigskip\bigskip
{\it Let us say a few words about the importance of the notion of PC domains.}
An old result of Skolem from 1939 states that an integral domain is a UFD if and only if it is atomic and GCD. However, weaker conditions were found which, together with atomicity, imply the UFD condition, namely, an integral domain is UFD if and only if it is atomic and AP (or pre-Schreier, or Schreier, or GCD). An ana\-lo\-gous situation is with the conditions which, together with atomicity, imply the PID condition (see the previous diagram). Cohn's 1968 theorem (\cite{c}) states an integral domain is PID if and only if it atomic and B\'ezout. The result od Chinh and Nam (\cite{cn}) states that an integral domain is a PID if and only if it is UFD and MIP, which is, as a consequence of our theorem \ref{main}, equivalent with atomic and MIP. {\it Our notion of PC domains provides a condition which is weaker than each of the conditions B\'ezout and MIP, however, it is still strong enough to be, together with atomicity, equivalent with the PID condition.} That is the main value of this notion.

\bigskip
We will now justify the previous diagram.

\medskip
\begin{proposition}\label{PC_AP}
Every PC domain is an AP domain.
\end{proposition}

\begin{proof}
Let $R$ be a PC domain and let $a$ be an atom of $R$. Suppose $a\mid xy$ for some $x,y\in R$, but $a\nmid x$ and $a\nmid y$. Then $x,y$ are not units. The ideal $(a,x)$ is proper, otherwise $ra+sx=1$ for some $r,s\in R$, hence $rya+sxy=y$, hence $rya+sta=y$ for some $t\in R$, hence $a\mid y$, a contradiction. Since $R$ is PC, there is a proper ideal $(b)$ containing $(a,x)$. But then $a\in (b)$, so $b\mid a$, hence (since $a$ is an atom and $b$ is a non-unit) $b\sim a$. Also $x\in (b)$, so $b\mid x$, hence $a\mid x$ (as $b\sim a$), a contradiction.
\end{proof}

\begin{proposition}\label{AP_not_PC}
There exists an AP domain which is not a PC domain.
\end{proposition}

\begin{proof}
Consider the additive monoid $M=\N_0\times \N_0$ and the associated monoid domain $R=F[X;M]$, where $F$ is a field. The polynomials $f\in R$ whose constant term is $0$  form a maximal ideal, say $\m$, of $R$. Let $D=R_\m$ be the localization of $R$ at $\m$. The elements of $D$ have the form
\begin{equation}\label{AP_not_PC_eq1}
x=\frac{X^{(r,s)}\cdot (a_0+a_1X^{(m_1,n_1)}\dots+a_kX^{(m_k,n_k)})}{1+b_1X^{(p_1, q_1)}+\dots+b_lX^{(p_l, q_l)}},
\end{equation}
where $k,l\ge 0$, $a_i, b_j\in F$ $(0\le i\le k, \, 1\le j\le l)$, $a_0\ne 0$, and $(m_1,n_1), \dots, (m_k,n_k)$  (pairwise distinct), $(p_1,q_1),\dots, (p_l,q_l)$ (pairwise distinct), $(r,s)$ are elements of $\N_0\times \N_0$. Hence $x\sim X^{(r,s)}$ and so the only atoms of $D$ are $X^{(0,1)}$ and  $X^{(1,0)}$, and they are both prime. Thus $D$ is an AP domain. Th ideal $(X^{(0,1)}, X^{(1,0)})$ is proper, but it is not contained in a proper principal ideal as no $X^{(r,s)}$ can divide both $X^{(1,0)}$ and $X^{(0,1)}$ unless it is a unit. Thus $D$ is not a PC domain.
\end{proof}

\begin{proposition}\label{PC_not_PS}
There exists a PC domain which is neither pre-Schreier (hence not B\'ezout), nor MIP.
\end{proposition}
\begin{proof}
Let $i$ be an irrational number such that $0<i<1$. Let $q$ be a rational number such that $19<q<20$. Consider the additive submonoid
\[M=([\,0, \,5+\frac{i}{2}\,]\cap \Q) \cup (5+\frac{i}{2},\,\infty)\]
of $\R_+$.
Since $\displaystyle{5<5+\frac{i}{2}<5.5}$, we have
\[8<q-10-i<10,\]
so that $q-10-i\in M$. Let $r$ be a rational number from $(10, 10+i)$. Then $8<q-r<10$. We claim that it is impossible to find four numbers $\alpha, \beta, \alpha', \beta'\in M$ such that the following relations hold (at the same time):
\begin{align}
\alpha+\beta &= 10+i,\label{star1}\\
\alpha+\alpha' &= r,\label{star2}\\
\beta+\beta' &= q-r.\label{star3}
\end{align}
Suppose to the contrary. Then by \ref{star2} at least one of the elements $\alpha, \alpha'$ is $\displaystyle{\le \frac{r}{2}}$, hence  $\displaystyle{<5+\frac{i}{2}}$, hence rational. Since $\alpha+\alpha'$ is rational, the other element is rational too. Thus $\alpha$ is rational. In the same way $\beta$ is rational. However, by the equation (\ref{star1}) $\alpha+\beta$ is irrational, a contradiction. 

Let now $R=F[X;M]$, where $F$ is a field. Then the polynomials $f\in R$ whose constant term is $0$ form a maximal ideal, say $\m$, of $R$. Let $D=R_\m$, the localization of $R$ at $\m$. The elements of $D$ have the form 
\[x=\frac{X^\gamma\,(a_0+a_1X^{\gamma_1}+\dots+a_mX^{\gamma_m})}{1+b_1X^{\delta_1}+\dots+b_mX^{\delta_m}},\]
where $m,n\ge 0$, $a_i, b_j\in F$ $(0\le i\le m$, $0\le j\le n$), and $\gamma, \gamma_1,\dots, \gamma_m$, $\delta_1, \dots, \delta_n$ are elements of $M$ with $0<\gamma_1<\dots<\gamma_m$, $0<\delta_1<\dots<\delta_n$. We can write $x=X^\gamma u$, where $u$ is a unit in $D$, $\gamma\in M$. The element $x$ is a unit if and only if $\gamma=0$. Since $q-10-i\in M$, we have
\begin{equation}\label{star4}
X^{10+i}\mid X^q=X^r\,X^{q-r}.
\end{equation}
We show that it is not possible to find two elements $y,z\in D$ such that $y\mid X^r$,\, $z\mid X^{q-r}$, and $yz=X^{10+i}$. Suppose to the contrary. Then we can assume $y=X^\alpha$ and $z=X^\beta$ for some $\alpha, \beta\in M$, such that there are $\alpha', \beta'\in M$ satisfying the relations (\ref{star1}), (\ref{star2}), and (\ref{star3}). However, we showed above that that is not possible. Hence $D$ is not pre-Schreier. In particular, $D$ is not B\'ezout.

Note that the maximal ideal $\m R_\m$ of $D$ is not finitely generated since for any $X^{\gamma_1},\dots, X^{\gamma_t}$, with $\gamma_i>0$ $(i=1,\dots, t)$ elements of $M$, there is a $\gamma\in M$ such that $0<\gamma<\min\{\gamma_1, \dots, \gamma_t\}$, so that $X^\gamma\notin (X^{\gamma_1}, \dots, X^{\gamma_t})$. Thus $D$ is not MIP. 

However, $D$ is a PC domain since for any $X^{\gamma_1}, X^{\gamma_2}\in D$ (with $\gamma_1, \gamma_2>0$ elements of $M$) there is a sufficiently small positive rational number $\gamma\in M$ such that $X^\gamma\mid X^{\gamma_1}$ and $X^\gamma\mid X^{\gamma_2}$. Hence $D\supset (X^\gamma)\supseteq (X^{\gamma_1}, X^{\gamma_2})$. 
\end{proof}

\begin{proposition}\label{MIP_not_PS}
There exists a MIP domain which is not pre-Schreier (hence not B\'ezout).
\end{proposition}
\begin{proof}
Let the numbers $i,q,r$, and the monoid $M$ be like in Proposition \ref{PC_not_PS}. Consider the submonoid 
\[N=(\Z\times (M\setminus\{0\})) \cup \N_0\]
of the additive monoid $\Z\times \R_+$. Let $R=F[X;N]$, where $F$ is a field. The polynomials $f\in R$ whose constant term is $0$  form a maximal ideal, say $\m$, of $R$. Let $D=R_\m$ be the localization of $R$ at $\m$. The elements of $D$ have the form
\begin{equation}\label{MIP_not_PS_eq1}
x=\frac{a_0X^{(k_0,\alpha_0)}+\dots+a_mX^{(k_m,\alpha_m)}}{1+b_1X^{(l_1,\beta_1)}+\dots+b_nX^{(l_n,\beta_n)}},
\end{equation}
where $m,n\ge 0$, $a_i, b_j\in F$ $(0\le i\le m, \, 1\le j\le n)$, and $(k_0,\alpha_0),\dots, (k_m,\alpha_m)$, $(l_1,\beta_1),\dots, (l_n,\beta_n)$ are elements of $N$. We assume that $\alpha_0\le \dots\le\alpha_m$ and the $(k_i,\alpha_i)$ are pairwise distinct, as well as that $0<\beta_1\le\dots\le\beta_n$ and the $(l_j,\beta_j)$ are pairwise distinct. Let $\nu$ be the largest element of $\{0,1,\dots,m\}$ such that $\alpha_0=\dots=\alpha_\nu$. Then we denote
\[x^\ast=a_0X^{(k_0,\alpha_0)}+\dots+a_\nu X^{(k_\nu,\alpha_0)}.\]
Note that for any $x,y\in D$ we have
\begin{equation}\label{MIP_not_PS_eq2}
(xy)^\ast=x^\ast\,y^\ast.
\end{equation}
Suppose also that $k_0<k_1<\dots<k_\nu$. We consider two cases.

\noindent
\underbar{1st case: $\alpha_0=0$.} Then we factor out $X^{(k_0,0)}$ from the numerator in (\ref{MIP_not_PS_eq1}) and have 
\begin{equation*}
x=(X^{(1,0)})^{k_0}\,\cdot\,\frac{a_0+a_1X^{(k_1-k_0,0)}+\dots+a_\nu X^{(k_\nu-k_0,\alpha_0)}+\dots+a_mX^{(k_m,\alpha_m)}}{1+b_1X^{(l_1,\beta_1)}+\dots+b_nX^{(l_n,\beta_n)}},
\end{equation*}
so that either 
\begin{equation}\label{MIP_not_PS_eq3}
x=u \quad\text{(if $k_0=0$),}
\end{equation}
or, 
\begin{equation}\label{MIP_not_PS_eq4}
x=(X^{(1,0)})^{k_0}\,u \quad\text{(if $k_0\ge 1$),}
\end{equation}
where $u$ is a unit in $D$.

\noindent
\underbar{2nd case: $\alpha_0>0$.} Then we factor out any $X^{(k,0)}$ $(k\in\N_0)$ from the numerator in (\ref{MIP_not_PS_eq1}) and we have 
\begin{equation}\label{MIP_not_PS_eq5}
x=(X^{(1,0)})^k\,\cdot\,\frac{a_0X^{(k_0-k,\alpha_0)}+\dots+a_mX^{(k_m-k,\alpha_m)}}{1+b_1X^{(l_1,\beta_1)}+\dots+b_nX^{(l_n,\beta_n)}}.
\end{equation}
Denote $\n=(X^{(1,0)})$, the ideal of $D$ generated by $X^{(1,0)}$. It follows from (\ref{MIP_not_PS_eq3}),  (\ref{MIP_not_PS_eq4}), and (\ref{MIP_not_PS_eq5}) that $\n=\m R_\m$, the maximal ideal of $D$, and that in the 1st case $x$ is an element of $\n^{k_0}\setminus \n^{k_0+1}$ $(k_0\ge 0)$, and in the 2nd case $x$ is an element of $\n^\omega=\cap_{k=1}^{\infty} \n^k$. Since the maximal ideal is principal, $D$ is a MIP domain. 

We now show that $D$ is not pre-Schreier. By (\ref{star4}) from Proposition \ref{PC_not_PS},
\begin{equation}\label{MIP_not_PS_eq6}
X^{(0,10+i)}\mid X^{(0,q)} = X^{(0,r)}\,X^{(0,q-r)}.
\end{equation}
We show that it is not possible to find two elements $y,z\in D$ such that
\begin{align}
y\mid \,& X^{(0,r)},\notag\\
z \mid\, & X^{(0,q-r)}, \notag\\
\,\,yz= \,&X^{(0,10+i)}.\label{MIP_not_PS_eq7}
\end{align}
Suppose to the contrary. Then
\begin{align}
yy' =&\, X^{(0,r)},\label{MIP_not_PS_eq8}\\
zz' =&\, X^{(0,q-r)},\label{MIP_not_PS_eq9}
\end{align}
for some $y',z'\in D$. Let $\alpha,\beta,\alpha',\beta'$ be the second coordinate of the exponents that appear in $y^\ast$, $z^\ast$, $y'^\ast$, and $z'^\ast$, respectively. Then  from (\ref{MIP_not_PS_eq7}), (\ref{MIP_not_PS_eq8}), and (\ref{MIP_not_PS_eq9}), using (\ref{MIP_not_PS_eq2}), we get 
\begin{align*}
\alpha+\beta &= 10+i,\\
\alpha+\alpha' &= r,\\
\beta+\beta' &= q-r.
\end{align*}
However, this is not possible as we have seen in the proof of Proposition \ref{PC_not_PS}.
\end{proof}

\begin{corollary}\label{PIP_MIP}
The MIP condition is strictly weaker than the PIP condition.
\end{corollary}
\begin{proof}
Otherwise every MIP domain would be a PID, hence pre-Schreier, contradicting the previous proposition.
\end{proof}


\begin{thebibliography}{99}
\bibitem{a}
ANDERSON, D.D.:
     \textit{GCD Domains, Gauss' Lemma, and Contents of Polynomials},
     in ``Non-Noetherian Commutative Ring Theory'' (S.~Chapman and S.~Glaz, Eds.), Springer Science+Business Media, B.V., Dordrecht, 2000, p. 1-31.
\bibitem{cn}       
CHINH, N.Q., NAM, P.H.:
     \textit{New Characterizations of Principal Ideal Domains},
     East-West J. Math., \textbf{10}(2008), 149-152.
\bibitem{c_book}     
COHN, P.M.:  
     \textit{Algebra I},
     Second Edition, John Wiley \& Sons Ltd., 1982.   
\bibitem{c}     
COHN, P.M.:
     \textit{B\'ezout rings and their subrings},
     Proc. Camb. Phil. Soc. \textbf{64}(1968), 251-264.    
\bibitem{df}
DUMMIT, D.S., FOOTE, R.M.:
     \textit{Abstract Algebra}, 3rd Ed., 
     John Wiley \$ Sons, Inc., 2004.
\bibitem{g}
GILMER, R.: 
    \textit{Commutative Semigroup Rings},
    The University of Chicago Press, Chicago, 1984.
\bibitem{k}
KAPLANSKY, I.: 
     \textit{Commutative Rings},
     Revised Edition, The University of Chicago Press, Chicago and London, 1974.    
\bibitem{n}
NORTHCOTT, D.G.: 
    \textit{Lessons on rings, modules and multiplicities},
    Cambridge University Press, Cambridge, 1968. 
\bibitem{z}
ZAFRULLAH, M.:
    \textit{On a property of pre-schreier domains},
    Commun. Algebra \textbf{15}(1987), 1895-1920.    
\end{thebibliography}
\end{document}